\documentclass[11pt]{amsart}
\usepackage{amssymb,amsmath,amsthm}
\usepackage[margin=1in]{geometry}
\usepackage[hidelinks]{hyperref}

\theoremstyle{plain}
\newtheorem{theorem}{Theorem}[section]
\newtheorem{proposition}[theorem]{Proposition}
\newtheorem{lemma}[theorem]{Lemma}

\theoremstyle{definition}

\theoremstyle{remark}
\newtheorem{remark}[theorem]{Remark}

\newcommand{\eps}{\varepsilon}
\newcommand{\N}{\mathbb{N}}
\newcommand{\Z}{\mathbb{Z}}
\newcommand{\Q}{\mathbb{Q}}
\newcommand{\R}{\mathbb{R}}
\newcommand{\pT}{p_T}
\newcommand{\lean}[1]{\texttt{#1}}

\makeatletter
\def\@settitle{\begin{center}%
  \baselineskip14\p@\relax
  \bfseries\Large
  \@title
  \end{center}%
}

\def\@setauthors{%
  \begingroup
  \def\thanks{\protect\thanks@warning}%
  \trivlist
  \centering
  \normalsize
  \@topsep30\p@\relax
  \advance\@topsep by -\baselineskip
  \item\relax
  \andify\authors
  \authors % <-- Removed \uppercasenonmath and \scshape
  \endtrivlist
  \endgroup
}
\makeatother

\title[The exponent $1/4$ in three complexity bounds]{The exponent $1/4$ in three complexity
bounds: integer bases, continued fractions, and the $(3/2)^n$ steering word}
\author{Ralf Stephan}
\thanks{Institute for Globally Distributed Open Research and Education (IGDORE) \\
MSC 11J68, 11A63, 11J87, 68R15}
\date{22 October 2026}

\begin{document}

\begin{abstract}
Let $\xi$ be a real algebraic irrational number and $p(n,\xi,b)$ the number of distinct blocks of
length $n$ in its base-$b$ expansion. Bugeaud and Evertse proved in 2008 that
$p(n,\xi,b)/(n(\log n)^\eta)$ is unbounded for every $\eta<1/11$. We prove this for every
$\eta<1/4$. We also prove the analogous bound for the continued fraction expansion of an algebraic
number of degree at least three, with every exponent $\delta<1/4$. Bugeaud announced this
statement in 2013 with an unspecified $\delta>0$. Finally, the subword complexity of the steering
word of $(3/2)^n$ exceeds $C\,k(\log k)^\eta$ infinitely often for every $\eta<1/4$ and every $C$.
The three proofs share one count: $\eps^{-1}$ classes times $\eps^{-3}$ subspaces from the
quantitative Subspace Theorem of Evertse and Ferretti. The subspaces are then disposed of at the
cost of a logarithm, with no further power of $\eps^{-1}$. All three results, and the quantitative
Subspace Theorem they rest on, are machine-verified in Lean~4.
\end{abstract}

\maketitle

\section{Introduction}

\subsection*{Borel's problem}
Borel introduced normal numbers in 1909 and proved that almost every real number is normal to
every integer base \cite{Bor09}. In 1950, at the end of a note on the decimal digits
of~$\sqrt2$, he wrote \cite{Bor50}:
\begin{quote}
\emph{En d\'efinitive, le probl\`eme de savoir si les chiffres d'un nombre tel que $\sqrt2$
satisfont ou non \`a toutes les lois que l'on peut \'enoncer pour des chiffres choisis au hasard
me para\^it toujours \^etre un des probl\`emes les plus importants qui se posent aux
math\'ematiciens.}
\end{quote}
It is now commonly expected that every real algebraic irrational number is normal to every integer
base. This is far out of reach. A much weaker measure of randomness is the \emph{block complexity}
\cite{BE08}: for an integer $b\ge2$ and a real number $\xi$ with $b$-ary expansion
$\xi=\lfloor\xi\rfloor+\sum_{j\ge1}a_jb^{-j}$,
\[
  p(n,\xi,b)=\operatorname{Card}\{a_{k+1}a_{k+2}\cdots a_{k+n}:k\ge0\}.
\]
A number normal to base $b$ has $p(n,\xi,b)=b^n$ for every $n$. The known lower bounds for
algebraic $\xi$ are barely above linear.
\begin{itemize}
\item By Morse and Hedlund \cite{MH38}, $p(n,\xi,b)\ge n+1$ for every irrational $\xi$.
\item Ferenczi and Mauduit \cite{FM97} used Ridout's theorem \cite{Rid57} to show
  $p(n,\xi,b)-n\to\infty$ for algebraic irrational $\xi$.
\item Adamczewski, Bugeaud and Luca \cite{ABL04} gave a combinatorial transcendence criterion
  based on the Schmidt Subspace Theorem. With it, Adamczewski and Bugeaud \cite{AB07} proved
  $p(n,\xi,b)/n\to\infty$. Since automatic sequences have complexity $O(n)$ \cite{Cob72}, this
  answers Cobham's question \cite{Cob68}: irrational automatic numbers are transcendental.
\item Bugeaud and Evertse \cite[Thm.~2.1]{BE08} brought in the quantitative Subspace Theorem of
  Evertse and Schlickewei \cite{ES02}. They proved that
  \begin{equation}\label{eq:BE}
    \limsup_{n\to\infty}\frac{p(n,\xi,b)}{n(\log n)^\eta}=+\infty\qquad\text{for every }\eta<1/11,
  \end{equation}
  so in particular $p(n,\xi,b)\ge n(\log n)^{0.09}$ for infinitely many $n$.
\end{itemize}

\subsection*{Continued fractions}
By Lagrange's theorem, the continued fraction expansion $\alpha=[0;a_1,a_2,\dots]$ is ultimately
periodic exactly when $\alpha$ is quadratic. Let
\[
  p(n,a)=\operatorname{Card}\{a_{\ell+1}\cdots a_{\ell+n}:\ell\ge0\}
\]
be the complexity function of the sequence of partial quotients of an algebraic $\alpha$ of
degree at least three. Its lower bounds came in the same order as for integer bases.
\begin{itemize}
\item Adamczewski and Bugeaud \cite{AB05} proved $p(n,a)-n\to\infty$.
\item Bugeaud \cite[Thm.~1.1]{Bug13} proved $p(n,a)/n\to\infty$. It follows that the continued
  fraction of an algebraic number of degree at least three is not automatic
  \cite[Thm.~1.2]{Bug13}.
\end{itemize}

\subsection*{The steering word}
Little is known about the distribution of the fractional parts of $(3/2)^n$ \cite{Mah68}. Write
$(3/2)^n=m_n+\eps_n$, where $m_n$ is the nearest integer, and put $t_n=2m_{n+1}-3m_n$. The
resulting word $T=(t_n)_{n\ge0}$ over $\{-2,\dots,2\}$ is the \emph{steering word} of $(3/2)^n$
\cite{St26}. Its subword complexity $\pT(k)$ satisfies $\pT(k)/k\to\infty$ \cite[Thm.~1.1]{St26}.
The proof transfers the Adamczewski--Bugeaud method to this word. Problem~5.2 of \cite{St26} asks
for a bound $\pT(k)\ge k(\log k)^\delta$ for infinitely many $k$, with some $\delta>0$.

\subsection*{What was announced}
Bugeaud remarked twice that the exponents could be improved or carried over, and neither remark
was followed up.
\begin{itemize}
\item \emph{Integer bases.} After restating \eqref{eq:BE} as Theorem~12.2 of his survey
  \cite{Bug15}, Bugeaud writes: \emph{``Using the recent results of [Evertse--Ferretti] allows us
  to show that Theorem~12.2 holds for $\eta$ in a slightly larger interval than $[0,1/11)$.''} No
  value and no proof are given.
\item \emph{Continued fractions.} At the end of \cite{Bug13}, \S6, Bugeaud writes that
  \emph{``it seems to be possible to prove''} that
  \begin{equation}\label{eq:61}
    \limsup_{n\to\infty}\frac{p(n,a)}{n(\log n)^\delta}=+\infty
  \end{equation}
  for some $\delta>0$, and that \emph{``more details will be given in a subsequent note''}. Two
  years later \cite[\S12]{Bug15} he says that the continued fraction analogues of Theorem~12.2
  and of an effective bound \emph{``have not been written yet, but there is little doubt that they
  hold''}.
\end{itemize}
Bugeaud's 2011 survey \cite{Bug11} and his monograph \cite[Ch.~8]{Bug12} restate \eqref{eq:BE}
as it is. Bugeaud and Evertse also note \cite[(2.4)]{BE08} that their method gives an extra factor
$(\log\log n)^{\delta}$ at $\eta=1/11$, without proof. We are not aware of any later work that
gives an exponent beyond $1/11$ in \eqref{eq:BE}, or that proves \eqref{eq:61} with an explicit
value of $\delta$. We checked the papers that cite \cite{BE08}.

\subsection*{Results}
We prove the following three theorems.

\begin{theorem}\label{thm:A}
Let $b\ge2$ be an integer and $\xi$ a real algebraic irrational number. For every real $\eta<1/4$
and every real $C$,
\[
  p(n,\xi,b)>C\,n(\log n)^\eta\qquad\text{for infinitely many }n.
\]
\end{theorem}

\begin{theorem}\label{thm:B}
Let $a=(a_\ell)_{\ell\ge1}$ be a sequence of positive integers such that
$\alpha=[0;a_1,a_2,\dots]$ is algebraic of degree at least three. For every $\delta$ with
$0\le\delta<1/4$ and every real $C$,
\[
  p(n,a)>C\,n(\log n)^\delta\qquad\text{for infinitely many }n.
\]
\end{theorem}

\begin{theorem}\label{thm:C}
For every real $\eta<1/4$ and every real $C$, the subword complexity of the steering word of
$(3/2)^n$ satisfies $\pT(k)>C\,k(\log k)^\eta$ for infinitely many $k$. In particular,
$\pT(k)\ge k(\log k)^{1/8}$ for infinitely many $k$, which solves Problem~5.2 of \cite{St26}.
\end{theorem}

Each statement is equivalent to $\limsup p(n)/(n(\log n)^\eta)=+\infty$. Theorem~\ref{thm:A}
improves \eqref{eq:BE} from $1/11$ to $1/4$. Theorem~\ref{thm:B} is \eqref{eq:61} with every
$\delta<1/4$. In Theorem~\ref{thm:B}, $p(n,a)$ may be infinite; this happens exactly when $a$ is
unbounded, and then there is nothing to prove.

\subsection*{The common count}
All three proofs follow \cite[\S\S6--7]{BE08}. Suppose $p(n)\le c\,n(\log n)^u$ for all large
$n$, and fix $v$ with $u<v<1/4$. A pigeonhole argument gives repetitions in the word, and these
give integer points $x_N,\dots,x_{2N-1}$. The heights of the points grow at least doubly
exponentially, and each point satisfies a system of Diophantine inequalities with exponent
$\eps=N^{-v+o(1)}$. Sort the points into about $\eps^{-1}$ classes according to the relative size
of two of their coordinates. Theorem~\ref{thm:EF} below then puts each class into
$\eps^{-3+o(1)}$ proper subspaces, so there are $\eps^{-4+o(1)}$ subspaces in all. The point of
this paper is that each of these subspaces contains only $O(\log\eps^{-1})$ of the points, up to
a remainder that lies on $\eps^{-3+o(1)}$ lines or planes holding $O(\log\eps^{-1})$ points each.
Hence
\[
  N\le\eps^{-4+o(1)}=N^{4v+o(1)},
\]
which is impossible for large $N$ since $4v<1$.

In \cite{BE08} the count is $\eps^{-11}$, from three factors:
\begin{itemize}
\item the $\eps^{-1}$ classes;
\item $\eps^{-7}$ subspaces per class, from \cite{ES02};
\item $\eps^{-3}$ points per subspace, from a two-variable quantitative Ridout theorem
  \cite[Cor.~5.2]{BE08}.
\end{itemize}
Using Evertse and Ferretti's count \cite{EF13} gives $1+3+3=7$, which is the ``slightly larger
interval'' of \cite{Bug15} (Remark~\ref{rem:seventh}). The remaining factor $\eps^{-3}$ goes away
as follows.
\begin{itemize}
\item In the integer-base and steering-word problems, a plane through the points either has a
  constant \emph{gap} (the difference of the two exponents) or it does not. If it does not,
  eliminating one variable leaves an inequality that Liouville's inequality or integer
  arithmetic bounds, with no Subspace Theorem at all. If the gap is constant, the plane carries a
  genuine Ridout problem. We count such planes only for points with a large gap. The points with
  a small gap are handled once, globally, by one more two-variable application. In that
  application each line holds at most one point (or three) because of a coprimality.
\item For continued fractions, a subspace that avoids one exceptional hyperplane is handled by a
  three-variable application at a fixed exponent $\delta=3/4$. This costs $O(1)$. The exceptional
  hyperplane is handled once, globally.
\end{itemize}

\subsection*{Formalization}
Theorems~\ref{thm:A}--\ref{thm:C} are proved in Lean~4 with Mathlib, using only the three
standard axioms of Lean. So are all the results they use, including the quantitative Subspace
Theorem in the form of Theorem~\ref{thm:EF} and Lemma~6.1 of \cite{BE08}. The code is in the
repository \url{https://github.com/rwst/Subspace-Theorems}. Appendix~\ref{app:lean} lists, for
every numbered statement below, the Lean declaration that proves it.

\section{Tools}

For $x\in\Z^n$ write $|x|=\max_i|x_i|$, and $|\cdot|_l$ for the $l$-adic absolute value
normalized by $|l|_l=l^{-1}$.

The following is the form of the quantitative Subspace Theorem we use. It follows from
Evertse and Ferretti's theorem for systems of inequalities \cite[\S3]{EF13}, specialized to the
ground field $\Q$, to linear forms with coefficients in $\Z+\Z\alpha+\Z\alpha^2$, and to the
coordinates at finitely many primes.

\begin{theorem}[Evertse--Ferretti \cite{EF13}]\label{thm:EF}
Let $\alpha$ be a real algebraic number, $S$ a finite set of primes, and $2\le n\le4$. There is
$K=K(\alpha,S,n)>0$ with the following property. Suppose we are given:
\begin{itemize}
\item linear forms $L_1,\dots,L_n$ in $X_1,\dots,X_n$, with coefficients
  $c_0+c_1\alpha+c_2\alpha^2$ where $c_j\in\Z$ and $|c_j|\le B$ for some $B\ge1$, and with
  $|\det(L_1,\dots,L_n)|\ge\gamma$ for some $0<\gamma\le1$;
\item reals $C\ge1$ and $0<\delta\le1$;
\item exponents $e_i\le1$ with some $e_i=1$ and some $e_i\le1-\delta/(2n)$;
\item exponents $f_{l,i}\le0$ for $l\in S$, with $f_{l,i}=0$ for some $i$, for each $l$;
\item and $\sum_ie_i+\sum_{l\in S}\sum_if_{l,i}\le-\delta$.
\end{itemize}
Then the points $x\in\Z^n\setminus\{0\}$ with
\begin{gather*}
  |L_i(x)|\le C|x|^{e_i},\qquad |x_i|_l\le|x|^{f_{l,i}}\qquad(1\le i\le n,\ l\in S),\\
  \log|x|\ge K\delta^{-1}(1+\log B+\log C+\log\gamma^{-1})
\end{gather*}
lie in at most $K\delta^{-3}(1+\log\delta^{-1})^2$ proper linear subspaces of $\Q^n$.
\end{theorem}

The exponent $3$ of $\delta^{-1}$ is what makes $1/4$ possible.
Theorem~\ref{thm:B} needs only $S=\emptyset$, and Theorem~\ref{thm:C} only $\alpha$ rational.

\begin{lemma}\label{lem:window}
Let $(t_n)$ be integers with $2t_n<t_{n+1}$ for all $n$, and let $R\ge1$. The indices
$n\ge n_0$ with $t_n\le R\,t_{n_0}$ number at most $\log R/\log2+1$.
\end{lemma}

\begin{proof}
$2^kt_{n_0}\le t_{n_0+k}$.
\end{proof}

We call the set in Lemma~\ref{lem:window} a \emph{window of ratio $R$} at $n_0$.

\section{Integer bases: proof of Theorem~\ref{thm:A}}\label{sec:A}

Let $b\ge2$ and let $\xi$ be real, algebraic, irrational, of degree $d$. We may assume
$0<\xi<1$. Suppose $p(n,\xi,b)\le c\,n(\log n)^u$ for all large $n$, with $0\le u<1/4$, and put
$v=(u+1/4)/2$. Liouville's inequality gives $A\ge1$ with
\begin{equation}\label{eq:liou}
  |c_0+c_1\xi|\ge(AZ^d)^{-1}\qquad\text{for }(c_0,c_1)\in\Z^2\setminus\{0\},\ |c_0|,|c_1|\le Z,\
  Z\ge1.
\end{equation}

\subsection{The points}
The following is Lemma~6.1 of \cite{BE08}.

\begin{lemma}[Bugeaud--Evertse]\label{lem:61}
There are $C\in\N$, integers $0\le r_n<t_n$ and $p_n\in\Z$ such that, for every $n$,
\begin{gather*}
  2t_n<t_{n+1},\qquad t_n\le(2n+2)^{C(n+1)},\qquad |p_n|\le b^{t_n},\qquad
  b\nmid p_n\text{ if }r_n\ge1,\\
  |b^{t_n}\xi-b^{r_n}\xi-p_n|\le(b^{t_n})^{-(\log t_n)^{-v}}.
\end{gather*}
\end{lemma}

\begin{remark}
\cite[Lemma~6.1]{BE08} is printed with the hypothesis $v<u$. It holds, and is used, for $v>u$.
This does not affect \cite[\S7]{BE08}.
\end{remark}

Put $x_n=(b^{t_n},b^{r_n},p_n)\in\Z^3$ and $s_n=t_n-r_n$, the \emph{gap} of $x_n$. Let
$\eps=(\log t_{2N})^{-v}$. Then $\eps^{-1}=N^{v+o(1)}$. For large $N$, $0<\eps\le1$, every $x_n$
with $N\le n<2N$ satisfies the inequality of Lemma~\ref{lem:61} with exponent $-\eps$, and
$\eps^{-1}$ is below any fixed multiple of $t_n\log b$. We count the $N$ points $x_n$,
$N\le n<2N$. Point $x_n$ has a \emph{large gap} if $\eps t_n<4s_n$, and a \emph{small gap} if
$4s_n\le\eps t_n$.

\subsection{Large gaps}
The first step is the three-variable step of \cite[\S7]{BE08}, with Theorem~\ref{thm:EF} in
place of \cite{ES02}.

\begin{proposition}\label{prop:planesA}
There is $K>0$ such that for $0<\eps\le1$ the points $(b^t,b^r,p)$ with $r<t$, $|p|\le b^t$,
$|b^t\xi-b^r\xi-p|\le(b^t)^{-\eps}$ and $K\eps^{-1}\le t\log b$ lie in at most
$K\eps^{-4}(1+\log\eps^{-1})^2$ proper subspaces of $\Q^3$.
\end{proposition}

\begin{proof}[Sketch]
Let $k=\lfloor2/\eps\rfloor+1$ and sort the points by $\ell=\lfloor kr/t\rfloor$. For
$l\mid b$ put $g_l=-v_l(b)\log l/\log b$, so that $|b^r|_l=(b^r)^{g_l}$ and
$\sum_{l\mid b}g_l=-1$. On the class $\ell$ apply Theorem~\ref{thm:EF} with $n=3$, $S$ the
primes dividing $b$, and $\delta=\eps-1/k\ge\eps/2$:
\begin{itemize}
\item at $\infty$, the forms $X_1$, $X_2$, $\xi X_1-\xi X_2-X_3$ with exponents $1$,
  $(\ell+1)/k$, $-\eps$;
\item at $l\mid b$, the coordinates with exponents $g_l$, $g_l\ell/k$, $0$.
\end{itemize}
The exponents sum to $1/k-\eps$. Summing over the $k\le3/\eps$ classes gives the bound.
\end{proof}

\begin{lemma}\label{lem:liou}
There is $D\ge1$ with the following property. Suppose $z_1b^t+z_2b^r+z_3p=0$ with $|z_i|\le Z$,
$Z\ge1$, $z_3\ne0$, $r<t$ and $|b^t\xi-b^r\xi-p|\le1$. Then $t\log b\le D(1+\log Z)$.
\end{lemma}

\begin{proof}
Put $s=t-r$. Multiply the approximation by $z_3$ and use $z_3p=-z_1b^t-z_2b^r$:
\[
  b^r\,\Gamma=z_3(b^t\xi-b^r\xi-p),\qquad \Gamma=(z_3\xi+z_1)b^s+(z_2-z_3\xi).
\]
So $b^r|\Gamma|\le Z$, and in particular $|\Gamma|\le Z$. Apply \eqref{eq:liou} to $z_1+z_3\xi$,
which is nonzero: this gives $b^s\le AZ^{d+1}(2+|\xi|)$. Next write
$\Gamma=(z_1b^s+z_2)+z_3(b^s-1)\xi$. Here $z_3(b^s-1)\ne0$ and both coefficients are at most
$W=2Zb^s$. Applying \eqref{eq:liou} again gives $b^r\le AW^dZ$. Multiply the two bounds and take
logarithms. This gives the claim with $D=(d+2)^2(1+\log A+\log(2+|\xi|))$.
\end{proof}

\begin{lemma}\label{lem:planeA}
There is $K$ such that, for $0<\eps\le1$, every proper subspace $U$ of $\Q^3$ contains at most
$K(1+\log\eps^{-1})$ of the points $x_n$ with a large gap.
\end{lemma}

\begin{proof}
Let $i_1<i_2$ be the first two indices with $x_n\in U$. The cross product
$z=x_{i_1}\times x_{i_2}$ satisfies $|z_i|\le Z=2b^{2t_{i_2}}$, and every $x_n\in U$ satisfies
$z_1b^{t_n}+z_2b^{r_n}+z_3p_n=0$. Moreover $(z_2,z_3)\ne(0,0)$. Otherwise $x_{i_1}$ and
$x_{i_2}$ would be proportional, which forces $b\mid p_{i_2}$ with $r_{i_2}\ge1$.

If $z_3\ne0$, Lemma~\ref{lem:liou} gives $t_n\le R\,t_{i_2}$ with $R=D(2+2\log b)/\log b$. By
Lemma~\ref{lem:window}, at most $\log R/\log2+2$ points lie in $U$.

If $z_3=0$, then $z_2\ne0$ and $b^{s_n}=-z_2/z_1$. So every point of $U$ has the same gap
$s_0\le t_{i_1}$. On such a plane the problem is a genuine Ridout problem, for
$(b^{s_0}-1)\xi$. A large gap gives $\eps t_n<4s_0\le4t_{i_1}$, a window of ratio $4/\eps$ at
$i_1$. There are at most $3+2\log\eps^{-1}$ such points.
\end{proof}

\subsection{Small gaps}

\begin{lemma}\label{lem:linesA}
There is $K$ such that, for $0<\eps\le1$, the points $y=(b^t-b^r,p)$ with $r<t$, $|p|\le b^t$,
$|b^t\xi-b^r\xi-p|\le(b^t)^{-\eps}$, $4(t-r)\le\eps t$ and $K\eps^{-1}\le t\log b$ lie on at most
$K\eps^{-3}(1+\log\eps^{-1})^2$ lines of $\Q^2$.
\end{lemma}

\begin{proof}
Apply Theorem~\ref{thm:EF} with $n=2$ and $\delta=\eps/2$.
\begin{itemize}
\item At $\infty$, take the forms $\xi X_1-X_2$ and $X_1$, with determinant $1$ and exponents
  $-\eps$ and $1$. Since $b^t/2\le|y|\le b^t$, we have $|\xi y_1-y_2|\le|y|^{-\eps}$.
\item At $l\mid b$, take the exponents $(1-\eps/4)g_l$ and $0$. Write $y_1=b^r(b^{s}-1)$. Since a
  small gap means $r\ge(1-\eps/4)t$, we get $|y_1|_l\le(b^r)^{g_l}\le|y|^{(1-\eps/4)g_l}$.
\end{itemize}
The exponents sum to $-\eps+1-(1-\eps/4)=-3\eps/4\le-\delta$.
\end{proof}

\begin{lemma}\label{lem:dvd}
Let $b\ne0$ be an integer. If $(b^{t_1}-b^{r_1})p_2=(b^{t_2}-b^{r_2})p_1$ with $r_1<t_1$ and
$r_1<r_2\le t_2$, then $b\mid p_2$.
\end{lemma}

\begin{proof}
Cancel $b^{r_1}$:
$(b^{t_1-r_1}-1)p_2=b^{r_2-r_1}(b^{t_2-r_2}-1)p_1$, and the right-hand side is divisible by $b$.
Since $b\cdot b^{t_1-r_1-1}-(b^{t_1-r_1}-1)=1$, $b$ is prime to $b^{t_1-r_1}-1$.
\end{proof}

No line contains two points $y_n,y_m$ ($n<m$) with a small gap. A small gap gives $3t\le4r$, so
$4r_m\ge3t_m>6t_n>4r_n$. Lemma~\ref{lem:dvd} then gives $b\mid p_m$ with $r_m\ge1$, contrary to
Lemma~\ref{lem:61}.

\subsection{The count}
By Proposition~\ref{prop:planesA}, Lemma~\ref{lem:planeA}, Lemma~\ref{lem:linesA} and the
remark above,
\[
  N\le K\eps^{-4}(1+\log\eps^{-1})^2\cdot K(1+\log\eps^{-1})+K\eps^{-3}(1+\log\eps^{-1})^2
  \le K'\eps^{-4}(1+\log\eps^{-1})^3=N^{4v+o(1)}.
\]
This contradicts $4v<1$ for large $N$, and proves Theorem~\ref{thm:A}.\qed

\section{Continued fractions: proof of Theorem~\ref{thm:B}}\label{sec:B}

Let $\alpha=[0;a_1,a_2,\dots]$ be algebraic of degree at least three, with convergents
$p_\ell/q_\ell$ ($p_{-1}=q_0=1$, $q_{-1}=p_0=0$). Suppose
$p(\ell,a)\le C\ell(\log\ell)^\delta$ for all large $\ell$, with $0\le\delta<1/4$. Then $p(1,a)$
is finite, so $a_\ell\le M$ for some integer $M\ge1$.

\subsection{The points}
We use Bugeaud's quadratic approximations of \cite[\S3]{Bug13}. Let $w\ge1$, $r\ge2$ and
$1\le u\le r$ be integers with
\[
  a_{m+r}=a_m\quad(w<m\le w+u),\qquad a_w\ne a_{w+r}\text{ if }w\ge2.
\]
The continued fraction completed by periodicity after $a_w$ represents a quadratic number. Its
minimal polynomial $v_0X^2-(v_1+v_2)X+v_3$ has the coefficients
\begin{align*}
  v_0&=q_{w-1}q_{w+r}-q_wq_{w+r-1}, & v_1&=q_{w-1}p_{w+r}-q_wp_{w+r-1},\\
  v_2&=p_{w-1}q_{w+r}-p_wq_{w+r-1}, & v_3&=p_{w-1}p_{w+r}-p_wp_{w+r-1}.
\end{align*}
Put $H=q_wq_{w+r}$ and $Q=q_w$, and consider the forms
\[
  L_1=\alpha^2X_0-\alpha(X_1+X_2)+X_3,\qquad L_2=\alpha X_0-X_1,\qquad L_3=\alpha X_0-X_2.
\]
Their values at $v$ are $P_n(\alpha)$ and the left-hand sides of (3.2) and (3.3) of
\cite{Bug13}.

\begin{lemma}\label{lem:cfpoint}
If $H^\eps\le2^u$, then $v=(v_0,\dots,v_3)$ satisfies:
\begin{enumerate}
\item $|v_i|\le H$ and $H\le(M+1)^4|v_0|$;
\item $1\le Q$ and $Q^2\le H$;
\item $|L_1(v)|\le24H^{-1-\eps}$;
\item $Q^2/(2H)\le|L_2(v)|\le\min(2Q^2/H,\,3H^{-\eps/2})$;
\item $|L_3(v)|\le2H/Q^2$;
\item \textup{(}Bugeaud's Claim\textup{)} for every $z\in\R^4$ with $z\cdot v=0$,
  \[
    Q\,|z_0+(z_1+z_2)\alpha+z_3\alpha^2|\le12(M+1)^2\textstyle\sum_i|z_i|.
  \]
\end{enumerate}
\end{lemma}

\begin{proof}[Sketch]
The upper bounds in (3)--(5) are (3.2)--(3.4) of \cite{Bug13}, together with the length $u$ of
the repetition. Inequality (1) uses $a_w\ne a_{w+r}$: the recurrence
$q_\ell=a_\ell q_{\ell-1}+q_{\ell-2}$ leaves in $v_0$ a term of size at least $H/(M+1)^3$. The
lower bound in (4) uses that the errors $q_\ell\alpha-p_\ell$ alternate in sign and satisfy
$q_{\ell+1}|q_\ell\alpha-p_\ell|\ge1/2$.
\end{proof}

\begin{lemma}\label{lem:config}
Suppose $a$ has at most $p$ factors of length $\ell\ge12$. Then there is a configuration as above
with $\ell\le4u$ and $w+r\le2p+\ell+1$.
\end{lemma}

\begin{proof}
Two of the positions $0,\dots,p$ carry the same factor of length $\ell$. This gives a repetition
of length at least $\ell/3$ at distance $r\ge u$. Slide it to the left as far as it goes.
\end{proof}

\subsection{The count}
Let $T_0=\{X_1=X_2\}\subset\Q^4$.

\begin{proposition}\label{prop:countB}
There is $K$ with the following property. Let $0<\eps\le1$ and let $v_k$ $(k\in I)$ be points as
in Lemma~\ref{lem:cfpoint}, with heights $H_k$ such that $\log H_k\ge1$,
$\log H_{k+1}\ge2\log H_k$ and $\log H_k\ge K\eps^{-K}$. Then
$\#I\le K\eps^{-4}(1+\log\eps^{-1})^3$.
\end{proposition}

\begin{proof}[Sketch]
The proof has four steps.

(i) \emph{First application.} Let $\rho=\log Q/\log H\in[0,1/2]$, and cover $[0,1/2]$ by at most
$10\eps^{-1}$ cells $[\rho^-,\rho^+]$ of length $\eps/4$. On one cell, $L_1,L_2,L_3,X_0$ are
bounded at $v$ by $24(M+1)^4|v|^{e_i}$, with
\[
  e=(-1-\eps,\ 2\rho^+-1,\ 1-2\rho^-,\ 1),\qquad \textstyle\sum_ie_i\le-\eps/2.
\]
The four forms have determinant $-1$. Theorem~\ref{thm:EF} with $n=4$, $S=\emptyset$ and
$\delta=\eps/2$ puts the points of a cell into $O(\eps^{-3}(1+\log\eps^{-1})^2)$ proper
subspaces of $\Q^4$. Altogether there are $O(\eps^{-4}(1+\log\eps^{-1})^2)$ subspaces.

(ii) \emph{Planes.} A subspace of dimension at most two contains $O(\log\eps^{-1})$ of the
points. Two of its points $v_a,v_b$ determine the others. Liouville's inequality for the
quadratic expression $L_1(x)L_2(y)-L_1(y)L_2(x)$ in $\alpha$ bounds the height of any further
point by $(cH_b^{4D+4})^{1/\eps}$. This is a window of ratio $O(\eps^{-1})$ in $\log H$, and
Lemma~\ref{lem:window} applies.

(iii) \emph{A proper subspace off $T_0$.} A proper subspace $U$ contains $O(\log\eps^{-1})$
points outside $T_0$. Let $v_a,v_b,v_c$ be the first three independent ones. Their cross product
$z$ is orthogonal to $U$ and has $|z|\le6H_c^3$. Since $v_a\notin T_0$,
$\Lambda=z_0+(z_1+z_2)\alpha+z_3\alpha^2\ne0$. By the Claim and the quadratic Liouville
inequality, $Q\le|z|^{O(1)}$. So every later point with $\log H\ge K(1+\log|z|)$ has
$Q^2\le H^{1/4}$.

Now eliminate a coordinate $x_j$ with $z_j\ne0$. The forms $z_jL_1,z_jL_2,z_jL_3$ in the
remaining three coordinates have determinant $\pm z_j^2\Lambda$ and exponents $(-1,-3/4,1)$. So
Theorem~\ref{thm:EF} with $n=3$ and the fixed value $\delta=3/4$ puts these points into $O(1)$
planes, each covered by (ii). The earlier points have $\log H\le9K\log H_c$, and there are $O(1)$
of them by Lemma~\ref{lem:window}.

(iv) \emph{The hyperplane $T_0$.} On $T_0$, the forms $L_1,L_2,X_0$ in $(X_0,X_1,X_3)$ have
determinant $1$ and exponents $(-1-\eps,0,1)$. Theorem~\ref{thm:EF} with $n=3$ and
$\delta=\eps$ puts the points of $T_0$ into $O(\eps^{-3}(1+\log\eps^{-1})^2)$ planes, each covered
by (ii).

Adding up gives $O(\eps^{-4}(1+\log\eps^{-1})^3)$.
\end{proof}

\begin{proof}[Proof of Theorem~\ref{thm:B}]
Take the lengths $\ell_k=E^{k+1}(k+1)^{2(k+1)}$ with $E$ large, so that $\log\ell_k\asymp k\log k$.
Lemma~\ref{lem:config} with $p=\lfloor C\ell_k(\log\ell_k)^\delta\rfloor$ gives configurations,
and so points $v_k$ with heights $H_k=q_wq_{w+r}$ satisfying
\[
  2^{\ell_k/8}\le H_k\le(M+1)^{2(w+r)}.
\]
For $E$ large these heights grow doubly exponentially. Choose $\eps_N$ with
$\eps_N^{-1}\asymp(\log\ell_{2N})^\delta\asymp(N\log N)^\delta$. Then $H_k^{\eps_N}\le2^u$ for
$N\le k<2N$, so Lemma~\ref{lem:cfpoint} applies with $\eps=\eps_N$. By
Proposition~\ref{prop:countB},
\[
  N\le K\eps_N^{-4}(1+\log\eps_N^{-1})^3=(N\log N)^{4\delta+o(1)},
\]
which fails for large $N$ since $4\delta<1$.
\end{proof}

\begin{remark}\label{rem:ninth}
The argument uses Theorem~\ref{thm:EF} only through its exponent $3$, at $n=3,4$. With a count
of $\delta^{-a}$ subspaces it gives every $\delta$ with $(a+1)\delta<1$. The count
$\delta^{-(n+4)}$ of Evertse--Schlickewei \cite{ES02} at $n=4$ gives $\delta<1/9$.
\end{remark}

\section{The steering word: proof of Theorem~\ref{thm:C}}\label{sec:C}

Write $(3/2)^n=m_n+\eps_n$ with $m_n$ the nearest integer, and $t_n=2m_{n+1}-3m_n$. A
\emph{repetition} of length $L$ at $a<c$ means $t_{a+i}=t_{c+i}$ for $0\le i<L$. By
\cite[Lemma~2.7]{St26}, a repetition forces $|\eps_c-\eps_a|\le(2/3)^L$. Put
$\lambda=\log2/\log3$. The numerical facts used are $\lambda\le2/3$, $2^{19}<3^{12}$,
$3^5<2^8$ and $3^{18}\ge648\cdot2^{18}$. Suppose $\pT(k)\le c\,k(\log k)^u$ for all large $k$,
with $0\le u<1/4$, and put $v=(u+1/4)/2$.

\subsection{The points}
The combinatorial half of \cite[Lemma~6.1]{BE08} applies to any word with finitely many factors
of each length.

\begin{lemma}\label{lem:violators}
There are $C\in\N$ and integers $r_n<\tau_n$ with
\[
  2\tau_n<\tau_{n+1},\qquad \tau_n\le(2n+2)^{C(n+1)},
\]
such that there is a repetition of length $\lceil\tau_n(\log\tau_n)^{-v}\rceil$ at
$(a_n,c_n)=(r_n+2,\tau_n+2)$.
\end{lemma}

Since $(2/3)^L\le3^{-L/3}$, a repetition of length $L\ge3c\eps$ at $a<c$ gives
$|\eps_c-\eps_a|\le(3^c)^{-\eps}$. The \emph{gap} of $(a,c)$ is $s=c-a$.

\subsection{Large gaps}
For $2\le a<c$, $s=c-a$ and $M=m_c-m_a\ge1$ put
\[
  x=(3^c,\ 2^s3^a,\ 2^cM)\in\Z^3,\qquad \mathcal N=x_0-x_1-x_2=2^c(\eps_c-\eps_a).
\]
Then $\mathcal N$ is odd, $|\mathcal N|<2^c$, and $|x|=3^c$.

\begin{proposition}\label{prop:planesC}
There is $K$ such that for $0<\eps\le1$ the points $x$ of the pairs with
$|\eps_c-\eps_a|\le(3^c)^{-\eps}$ and $K\eps^{-1}\le c\log3$ lie in at most
$K\eps^{-4}(1+\log\eps^{-1})^2$ planes.
\end{proposition}

\begin{proof}
Let $k=\lfloor2/\eps\rfloor+1$ and sort the points by $\ell=\lfloor ks/c\rfloor$. On one class
apply Theorem~\ref{thm:EF} with $n=3$, $S=\{2,3\}$, $C=1$ and the exponents
\[
\begin{array}{c|ccc}
 & X_0-X_1-X_2 & X_1 & X_2\\\hline
 \infty & \lambda-\eps & 1-(1-\lambda)\ell/k & 1\\
 2 & 0 & -\lambda\ell/k & -\lambda\\
 3 & -1 & -(1-(\ell+1)/k) & 0
\end{array}
\]
At $2$ and $3$ the forms are the coordinates. The entries sum to $1/k-\eps=-\delta$ with
$\delta\ge\eps/2$. Each inequality is an identity between powers of $3^c$. Summing over the
$k\le3/\eps$ classes gives the bound.
\end{proof}

Fix a plane $w_0X_0+w_1X_1+w_2X_2=0$.

\begin{lemma}\label{lem:sameGap}
If $w_2\ne0$, any two points $(a_1,c_1)$, $(a_2,c_2)$ on the plane with the same gap satisfy
$7c_2<19c_1$.
\end{lemma}

\begin{proof}
Eliminating $X_2$ gives $\mathcal N=\beta_13^c+\beta_22^s3^a$ with $\beta_1,\beta_2\in\Q$ fixed by
the plane. So $\mathcal N_i=3^{a_i}(\beta_13^s+\beta_22^s)$ and
$\mathcal N_23^{a_1}=\mathcal N_13^{a_2}$. If $a_2>a_1$, then $3^{c_2-c_1}\le|\mathcal N_2|<2^{c_2}$,
because $\mathcal N_1$ is odd. If moreover $7c_2\ge19c_1$, then $3^{12c_2}\le3^{19(c_2-c_1)}
<2^{19c_2}$, which is false.
\end{proof}

\begin{lemma}\label{lem:threeTerms}
If $n_1\ne0$, $|n_2|\le m$ and $z=3^sn_1-2^sn_2\ne0$, then $3^s\le2(2m+1)^3|z|$.
\end{lemma}

\begin{proof}
If $2^{s+1}m\le3^s$ then $|z|\ge3^s/2$. Otherwise $3^s<(2m+1)2^s$, and cubing gives
$3^s<(2m+1)^3$.
\end{proof}

\begin{lemma}\label{lem:distinctGaps}
Let $w_2\ne0$. If two points $(a_1,c_1)$, $(a_2,c_2)$ on the plane have different gaps, then
every point $(a,c)$ on it has $c<18(c_1+c_2)+18$.
\end{lemma}

\begin{proof}
Eliminate $\beta_1,\beta_2$. This gives $\mathcal ND=3^a(3^sn_1-2^sn_2)$, where $D$, $n_1$ and
$n_2$ are minors built from the two given points. Comparing $2$-adic valuations gives $D\ne0$ and
$n_1\ne0$. Lemma~\ref{lem:threeTerms} gives $3^c<250\cdot648^{c_1+c_2}2^c$, which fails for
$c\ge18(c_1+c_2)+18$.
\end{proof}

\begin{lemma}\label{lem:planeC}
Each plane carries at most $12(1+\log\eps^{-1})$ indices $n$ of Lemma~\ref{lem:violators} with
a large gap, $\eps c_n<4s_n$.
\end{lemma}

\begin{proof}
If $w_2=0$, then $w_03^s+w_12^s=0$, so the gap is a constant $s_0$. A large gap gives
$\tau_n<4s_0/\eps\le(4/\eps)\tau_{n_1}$, where $n_1$ is the first index on the plane. This is a
window of ratio $4/\eps$.

If $w_2\ne0$, Lemma~\ref{lem:sameGap} confines the indices with the gap of $n_1$ to a window of
ratio $8$. Let $n'$ be the first index with another gap. Lemma~\ref{lem:distinctGaps} confines
the remaining indices to a window of ratio $128$ at $n'$. Lemma~\ref{lem:window} finishes the
proof.
\end{proof}

\subsection{Small gaps}

\begin{lemma}\label{lem:linesC}
The points $y=(3^c-2^s3^a,\ 2^cM)$ of the pairs with small gap $s\le\eps c/4$ lie on at most
$K\eps^{-3}(1+\log\eps^{-1})^2$ lines, and each line carries at most three indices $n$ of
Lemma~\ref{lem:violators}.
\end{lemma}

\begin{proof}
The coordinate $y_0=3^a(3^s-2^s)$ is odd and divisible by $3^a$. Also $3^{c-1}\le|y|\le3^c$ and
$y_0-y_1=\mathcal N$. Apply Theorem~\ref{thm:EF} with $n=2$, $C=3$, $\delta=3\eps/4$, and the
exponents
\[
\begin{array}{c|cc}
 & Y_0-Y_1 & Y_1\\\hline
 \infty & \lambda-\eps & 1\\
 2 & 0 & -\lambda\\
 3 & -(1-\eps/4) & 0
\end{array}
\]
The small gap enters at $3$: $|y_0|_3\le3^{-a}\le(3^c)^{-(1-\eps/4)}$. Two points on one line
satisfy $y_0y_1'=y_0'y_1$. Since $y_0'$ is odd, $2^{c'}\mid y_1<3^c$, and $3^5<2^8$ gives
$5c'<8c$, hence $\tau'<4\tau$.
\end{proof}

\subsection{The count}
Let $\eps_0=(\log\tau_{2N})^{-v}$ and $\eps=\eps_0/6$. For $N\le n<2N$, the pair $(a_n,c_n)$
satisfies $|\eps_{c_n}-\eps_{a_n}|\le(3^{c_n})^{-\eps}$, and for large $N$ the threshold of
Proposition~\ref{prop:planesC} holds. Splitting the $N$ indices by the size of the gap,
Proposition~\ref{prop:planesC} and Lemmas~\ref{lem:planeC} and~\ref{lem:linesC} give
\[
  N\le K\eps^{-4}(1+\log\eps^{-1})^2\cdot12(1+\log\eps^{-1})+3K\eps^{-3}(1+\log\eps^{-1})^2
  =N^{4v+o(1)}.
\]
This contradicts $4v<1$ and proves Theorem~\ref{thm:C}.\qed

\section{Remarks on the exponents}

\begin{remark}\label{rem:seventh}
If the original argument of \cite[\S\S6--7]{BE08} is run with Theorem~\ref{thm:EF} in place of
\cite{ES02}, both the three-variable step and the per-plane step improve. This gives
\eqref{eq:BE} for every $\eta<1/7$, since $7=1+3+3$. It is the most direct reading of Bugeaud's
``slightly larger interval than $[0,1/11)$'' \cite{Bug15}. This intermediate statement is also
machine-verified (Appendix~\ref{app:lean}).
\end{remark}

\begin{remark}
In each proof, Ridout's theorem is still needed on the planes of constant gap. There we lose only
a logarithm, because the large-gap condition $\eps t<4s$ confines the points to a window of
ratio $4/\eps$. Points with a small gap behave like approximations to a single number, and they
are counted once on lines. Within that count, the coprimality of $b$ and $b^s-1$ (or the oddness
of $3^s-2^s$) separates the points. In Theorem~\ref{thm:B} the exceptional hyperplane $T_0$
plays the role of the small gaps.
\end{remark}

\begin{remark}
The exponent $1/4$ is $1/(1+3)$: one power of $\eps^{-1}$ for the classes, three for the count of
Theorem~\ref{thm:EF}. Improving it by this method would need either fewer classes or a
count better than $\delta^{-3}$ in the quantitative Subspace Theorem.
\end{remark}

\section*{Acknowledgements}
We thank Yann Bugeaud for drawing our attention to this question. This work relies on a complete
formalization of the Subspace Theorem and its quantitative versions, which is publicly available
and is being incorporated into the TauCeti library.

\appendix
\section{The Lean formalization}\label{app:lean}

Everything is in the repository
\begin{center}\url{https://github.com/rwst/Subspace-Theorems},\end{center}
a Lean~4 project on Mathlib. Every declaration below is proved without \texttt{sorry}. Its axioms, as
reported by \texttt{\#print axioms}, are the three standard axioms \texttt{propext},
\texttt{Classical.choice} and \texttt{Quot.sound}. The repository declares no axioms of its own.
The steering-word theorems are additionally certified by the \texttt{comparator} tool against an
independent statement file. The directories are:
\begin{itemize}
\item \texttt{BugeaudEvertse2008/} for \cite{BE08}, and its subdirectory \texttt{Additions/} for
  Section~\ref{sec:A};
\item \texttt{Bugeaud2013/Additions/} for Section~\ref{sec:B};
\item \texttt{Stephan2026a/} for Section~\ref{sec:C};
\item \texttt{QuantitativeSubspace/} and \texttt{DiophantineApproximation/} for the shared tools.
\end{itemize}

In the Lean statements:
\begin{itemize}
\item Theorem~\ref{thm:A} is stated for $\xi=\sum_{j\ge0}a_jb^{-j-1}$ with digits
  $a_j\in\{0,\dots,b-1\}$, irrational and algebraic. This is the case $0<\xi<1$, which suffices.
\item Theorem~\ref{thm:B} counts the factors of $a_1a_2\cdots$ as a cardinal, possibly infinite.
\item Theorem~\ref{thm:C} is stated for the word $t_n=2m_{n+1}-3m_n$ with $m_n$ the rounded value
  of $(3/2)^n$.
\end{itemize}

\medskip
\begin{center}
\footnotesize
\begin{tabular}{lll}
\hline
Statement & Lean declaration & File\\\hline
\multicolumn{3}{l}{\emph{Section~\ref{sec:A}: \lean{BugeaudEvertse2008/}}}\\
Thm.~\ref{thm:A} & \lean{Real.frequently\_lt\_complexity\_of\_lt\_quarter} & \lean{Additions/Theorem21Quarter}\\
Lemma~\ref{lem:61} & \lean{Real.exists\_approx\_of\_complexity\_le} & \lean{Repetitions}\\
Prop.~\ref{prop:planesA} & \lean{Real.exists\_planes\_of\_approx} & \lean{ThreeVariable}\\
Lemma~\ref{lem:liou} & \lean{Real.exists\_mul\_log\_le\_of\_plane} & \lean{Additions/PlaneLiouville}\\
Lemma~\ref{lem:planeA} & \lean{Real.card\_mem\_plane\_le\_log} & \lean{Additions/PlaneLiouville}\\
Lemma~\ref{lem:linesA} & \lean{Real.exists\_lines\_of\_small\_gap} & \lean{Additions/SmallGap}\\
Lemma~\ref{lem:dvd} & \lean{Int.dvd\_of\_pow\_sub\_pow\_mul\_eq} & \lean{Additions/SmallGap}\\
Remark~\ref{rem:seventh} & \lean{Real.frequently\_lt\_complexity} & \lean{Theorem21}\\
\multicolumn{3}{l}{\emph{Section~\ref{sec:B}: \lean{Bugeaud2013/Additions/}}}\\
Thm.~\ref{thm:B} & \lean{Nat.frequently\_lt\_encard\_factors\_of\_lt\_quarter} & \lean{Delta}\\
Lemma~\ref{lem:cfpoint} & \lean{Nat.cfPoint\_spadeVec} & \lean{CFPoints}\\
Lemma~\ref{lem:config} & \lean{Nat.exists\_config} & \lean{Complexity}\\
Prop.~\ref{prop:countB} & \lean{Real.card\_le\_of\_cfPoints} & \lean{Count}\\
\quad step (i) & \lean{Real.exists\_cover\_four} & \lean{Count}\\
\quad step (ii) & \lean{Real.card\_mem\_le\_of\_finrank\_le\_two} & \lean{Points}\\
\quad step (iii) & \lean{Real.card\_mem\_le\_of\_ne\_top} & \lean{Subspaces}\\
\quad step (iv) & \lean{Real.exists\_planes\_t0} & \lean{Hyperplanes}\\
Remark~\ref{rem:ninth} & \lean{Nat.frequently\_lt\_encard\_factors\_of\_lt\_ninth} & \lean{Delta}\\
\multicolumn{3}{l}{\emph{Section~\ref{sec:C}: \lean{Stephan2026a/}}}\\
Thm.~\ref{thm:C} & \lean{SteeringWord.frequently\_lt\_complexity} & \lean{QuantitativeTheorem}\\
\quad ($1/8$) & \lean{SteeringWord.quantitativeSuperlinear} & \lean{QuantitativeTheorem}\\
Lemma~\ref{lem:violators} & \lean{SteeringWord.exists\_violators\_of\_complexity\_le} & \lean{QuantitativeViolators}\\
Prop.~\ref{prop:planesC} & \lean{SteeringWord.exists\_planes\_of\_violators} & \lean{QuantitativePlanes}\\
Lemma~\ref{lem:sameGap} & \lean{SteeringWord.seven\_mul\_lt\_of\_same\_gap} & \lean{QuantitativePlaneStep}\\
Lemma~\ref{lem:threeTerms} & \lean{SteeringWord.three\_pow\_le\_of\_ne\_zero} & \lean{QuantitativePlaneStep}\\
Lemma~\ref{lem:distinctGaps} & \lean{SteeringWord.lt\_of\_ne\_gap} & \lean{QuantitativePlaneStep}\\
Lemma~\ref{lem:planeC} & \lean{SteeringWord.card\_mem\_plane\_le} & \lean{QuantitativePlaneStep}\\
Lemma~\ref{lem:linesC} & \lean{SteeringWord.exists\_lines\_of\_violators} & \lean{QuantitativeLines}\\
 & \lean{SteeringWord.card\_mem\_line\_le} & \lean{QuantitativeLines}\\
\multicolumn{3}{l}{\emph{Shared tools: \lean{QuantitativeSubspace/}, \lean{DiophantineApproximation/}}}\\
Thm.~\ref{thm:EF} & \lean{Real.subspaceBoundAt\_three} & \lean{RealFormCount}\\
Lemma~\ref{lem:window} & \lean{Nat.card\_le\_of\_doubling\_window} & \lean{DoublingSequences}\\
\hline
\end{tabular}
\end{center}

% omitted: Real.exists_plane_eq (cross-product reduction, folded into Lemma 3.5),
% Real.exists_inv_le_abs_add_mul (Liouville (3.1)), Real.card_mem_plane_le (BE08 per-plane
% Ridout count, used only for 1/7), Real.eventually_doubling_thresholds and
% Real.eventually_thresholds (asymptotics of eps), Bugeaud2013 Asymptotics.lean (lenSeq lemmas),
% Liouville2.lean (quadratic Liouville), Forms.lean (projCoef/t0Coef/crossZ determinants),
% Real.exists_planeCover, Real.lamCoef_ne_zero, Real.rpow_le_of_mem_span_pair (folded into the
% sketch of Prop. 4.3), Real.SubspaceBound (S = empty case of Thm 2.1), Nat.frequently_lt_encard_factors
% (general exponent a, stated in Remark 4.4), SteeringWord kernel-point arithmetic
% (kernelNum_eq, odd_kernelNum, abs_kernelNum_lt, iSup_abs_kernelPoint, padicNorm lemmas,
% odd_three_pow_sub, five_mul_lt_of_mul_eq, two_pow_mul_ne*), abs_eps_sub_le_rpow (stated in
% prose after Lemma 5.1), complexity_shift_le, factors_shift_subset, private helpers.

\end{document}